\documentclass[12pt]{amsart}
\usepackage{verbatim}
\usepackage{amssymb}
\usepackage{amsbsy}
\usepackage{amscd}
\usepackage{amsmath}
\usepackage{amsthm}
\usepackage{amsxtra}
\usepackage{latexsym}
\usepackage[mathscr]{eucal}
\usepackage{upgreek}
\usepackage{wasysym}
\usepackage[all,cmtip]{xy}
\usepackage[dvipdfm, colorlinks]{hyperref}

\newcommand{\g}{\mathfrak{g}}

\newcommand{\mc}{\mathcal}
\newcommand{\mf}{\mathfrak}

\newcommand{\on}{\operatorname}

\newcommand{\W}{\mathscr{W}}

\newcommand{\affg}{\widehat{\mathfrak{g}}}

\newcommand{\fing}{\mathfrak{g}}

\newcommand{\BGG}{{\mathcal O}}

\newcommand{\N}{\mathbb{N}}
\newcommand{\Q}{\mathbb{Q}}

\newcommand{\bra}{{\langle}}
\newcommand{\ket}{{\rangle}}

\newcommand{\lam}{\lambda}
\newcommand{\ra}{\rightarrow}

\newcommand{\Z}{\mathbb{Z}}

\newcommand{\cprime}{$'$}

\renewcommand{\*}{{\otimes}}
\newcommand{\C}{\mathbb{C}}

\theoremstyle{plain}
\newtheorem{Th}{Theorem}[section]

\newtheorem{Pro}[Th]{Proposition}

\newtheorem{Lem}[Th]{Lemma}

\theoremstyle{definition}

\theoremstyle{remark}
\newtheorem{Def}[Th]{Definition}
\newtheorem{Rem}[Th]{Remark}

\DeclareMathOperator{\End}{End}
\DeclareMathOperator{\gr}{gr}

\DeclareMathOperator{\ad}{ad}

\title[Quasi-lisse vertex algebras and MLDE]{Quasi-lisse vertex algebras and  modular linear differential equations}
\author{Tomoyuki Arakawa}
\address{Research Institute for Mathematical Sciences, Kyoto University,
 Kyoto 606-8502 JAPAN} 
\email{arakawa@kurims.kyoto-u.ac.jp}
\address{Department of Mathematics, MIT 
77 Massachusetts Ave 
Cambridge MA 02139 USA}

\author{Kazuya Kawasetsu}
\address{School of Mathematics and Statistics, The University of Melbourne, 3010 AUSTRALIA}
\email{kazuya.kawasetsu@unimelb.edu.au}

\subjclass[2010]{17B69, 17B67, 11F22, 81R10}

 \keywords{Vertex algebras, Modular linear differential equations, Quasimodular forms, Affine Kac-Moody algebras, Affine $\mathcal{W}$-algebras, 
 Associated varieties, Deligne exceptional series,
 Schur limit of superconformal index}

\begin{document}

\maketitle

\begin{center}
{\small {\em Dedicated to the great mathematician 	Bertram Kostant }}
\end{center}

\begin{abstract}
We introduce a notion of quasi-lisse vertex algebras,
which generalizes admissible affine vertex algebras.
We show that the normalized character of 
an ordinary module over 
a  quasi-lisse vertex operator algebra
has a modular invariance property,
in the sense that it
satisfies  a modular linear differential equation.
As an application
we obtain the explicit character formulas
of simple affine vertex algebras associated with the Deligne exceptional series
at level $-h^{\vee}/6-1$,
which express
the homogeneous Schur 
indices of 
4d SCFTs
studied by Beem,
Lemos,
Liendo,
Peelaers,
Rastelli and van Rees,
as quasi-modular forms.

\end{abstract}


\section{Introduction}
The vertex algebra $V$ is called {\em lisse}, or {\em $C_2$-cofinite},
if the dimension of the associated variety $X_V$ is zero.
For instance,  a simple affine vertex algebra $V$ associated with an affine Kac-Moody algebra $\widehat{\mathfrak{g}}$
is lisse  if and only if $V$ is an integrable representation as a $\widehat{\mathfrak{g}}$-module.
Thus, the lisse property generalizes the integrability condition to an arbitrary vertex algebra.

It is known that a lisse vertex operator algebra $V$ has nice properties, such as 
the modular invariance of characters \cite{Z,Miy04},
and most 
theories of vertex operator algebras have been  build under this finiteness condition
(see e.g.\  \cite{DLM00,Hua08}).
However, there do exist significant vertex algebras that do not satisfy the lisse condition.
For instance,
admissible affine vertex algebras 
do not satisfy the lisse condition unless they are integrable,
but nevertheless their representations are semisimple in category $\BGG$ (\cite{AdaMil95,A12-2})
and have the modular invariance property (\cite{KacWak89,AEkeren}).
Moreover, there are a huge number of
vertex algebras 
constructed in \cite{BeeLemLie15} from 
 four dimensional $N=2$ superconformal field theories (SCFTs),  whose character coincides with  the Schur limit of the superconformal index of the
 corresponding
 four dimensional theories.
These vertex algebras do not
  satisfy the lisse property in general either.

 In this paper we propose the {\em quasi-lisse} condition
 that generalizes the lisse condition.
More precisely,
we call a conformal vertex algebra $V$ quasi-lisse
if 
its associated variety $X_V$ has finitely many symplectic leaves.
For instance,
a simple affine vertex algebra $V$ associated with  $\affg$ is quasi-lisse if and only if $X_V$ is contained in the nilpotent cone $\mc{N}$ of $\fing$.
Therefore,
by \cite{FeiMal97,Ara09b},
all the admissible affine vertex algebras
are quasi-lisse.
Moreover,
the $W$-algebras obtained by quasi-lisse affine vertex algebras by the quantized Drinfeld-Sokolov  reduction (\cite{FF90,KacRoaWak03}) is quasi-lisse as well.
The vertex algebras 
 constructed from 
 4d SCFTs are also expected to be quasi-lisse, 
 since their associated varieties conjecturally 
 coincide with the Higgs branches of the corresponding four dimensional theories (\cite{R}).

We show that the normalized character of an ordinary representation of a  quasi-lisse vertex operator algebra
has a modular invariance property,
in the sense that it
satisfies  a {\em modular linear differential equation} (MLDE) (cf.~\cite{MMS}, \cite{KZ},
\cite{Mas07}, {\cite{Mil}, \cite{KNS} and~\cite{AKNS}).
This seems to be new even for an admissible affine vertex algebra.
Moreover,
using MLDE,
we obtain the explicit character formulas 
of simple affine vertex algebras associated with the Deligne exceptional series
$
A_1\subset A_2 \subset G_2 \subset D_4 \subset F_4 \subset E_6 \subset E_7 \subset E_8
$ (\cite{D})
at level $-h^{\vee}/6-1$.
These vertex algebras arose
in \cite{BeeLemLie15}
as $2d$ chiral algebras 
constructed from 
4d SCFTs\footnote{For types $G_2$ and $F_4$ the connection with 
4d SCFTs is conjectural.}.
Thus \cite{BeeLemLie15}, our result expresses  
 the homogeneous Schur indices of
 the corresponding 4d SCFTs as  {\em (quasi)modular forms}. 
This result is rather surprising
especially for types $D_4$,
$E_6$,
$E_7$ and $E_8$ (non-admissible casees), since the characters of these vertex algebras
are
written \cite{KasTan00} in terms of {\em non-trivial} Kazhdan-Lusztig polynomials 
as their highest weights are not regular dominant.

We note that in \cite{CorSha16}
the authors have obtained a  conjectural expression
of Schur indices in terms of 
Kontsevich-Soibelman
wall-crossing invariants,
which 
we hope to investigate  in future works.
\smallskip
\subsection*{Acknowledgments}
The first named author  thanks 
Victor Kac,
Anne Moreau,
Hiraku Nakajima,
Takahiro Nishinaka,
Leonardo 
  Rastelli,
  Shu-Heng Shao,
 Yuji Tachikawa
  and Dan Xie
  for  valuable discussion.
  He thanks
  Christopher Beem
  for pointing out an error 
  in the first version of this article.
  Some part of this work was done while he was visiting
  Academia Sinica, Taiwan, in August 2016, for
  ^^ ^^ Conference in Finite Groups and Vertex Algebras".
  He thanks  the organizers of the conference. 
  He is partially  supported 
by JSPS KAKENHI Grant Number
No.\ 20340007 and No.\ 23654006.
The second named author would like to thank Hiroshi Yamauchi for the helpful advice.
He was partially supported by
 JSPS KAKENHI Grant Number No.\ 14J09236.

This research was supported in part by Perimeter Institute for Theoretical Physics. Research at Perimeter Institute is supported by the Government of Canada through Industry Canada and by the Province of Ontario through the Ministry of Economic Development \& Innovation.
\section{Quasi-lisse vertex algebras}
Let $V$ be a conformal vertex algebra, 
$R_V=V/C_2(V)$  
the Zhu's $C_2$-algebra of $V$ (\cite{Z}),
where
 $C_2(V)=\bra a_{(-2)}b\mid a,b\in V\ket_{\C}$.
 The space $R_V$ is a Poisson algebra
 by
 \begin{align*}
\bar a\cdot \bar b=\overline{a_{(-1)}b},
\quad 
\{\bar a,\bar b\}=\overline{a_{(0)}b}.
\end{align*}
Here $\bar a$ denotes the image of $a\in V$ in $R_V$,
and $a(z)=\sum_{n\in \Z}a_{(n)}z^{-n-1}\in (\End V)[[z,z^{-1}]]$
is the quantum field corresponding to $a\in V$.
In this paper we assume $V$ is finitely strongly generated, that is,
$R_V$ is finitely generated.

The associated variety \cite{Ara12} of a vertex algebra $V$ is the finite-dimensional
algebraic variety defined
by 
$$X_V=\on{Specm}(R_V).$$ 
Since $R_V$ is a Poisson algebra, 
we have a finite partition $$X_V=\bigsqcup_{k=0}^rX_k,$$
where $X_k$ are smooth analytic Poisson varieties (see e.g. \cite{Brown:2003kq}).
Thus, for any point $x\in X_k$ there is a well defined symplectic leaf $\mathscr{L}_x\subset X_k$
through it.
\begin{Def}
A  finitely strongly generated  vertex algebra $V$ is called {\em quasi-lisse}
if $X_V$ has only finitely many symplectic leaves.
\end{Def}
Let $V$ be a quasi-lisse vertex algebra.
The finiteness of the symplectic leaves implies \cite{Brown:2003kq}
 that
 the symplectic leaf $\mathscr{L}_x$ 
 at $x\in X_V$
coincides with the regular locus of the zero variety of the maximal Poisson ideal 
contained in 
the maximal ideal corresponding to $x$.
Thus,  each leaf $\mathscr{L}_x$ is a smooth connected locally-closed {\em algebraic} subvariety in $X_V$.
In particular,
every irreducible component of $X_V$ is the closure of a symplectic leaf (\cite[Corollary 3.3]{Gin03}).

For us,  the importance  of the finiteness of the symplectic leaves 
 is in the following fact that has been established by Etingof and Schedler.
 \begin{Th}[\cite{EtiSch10}]\label{Th:ES}
Let $R$ be a finitely generated Poisson algebra,
and suppose that $\on{Specm}(R)$ has finitely many symplectic leaves.
Then
$$\dim R/\{R,R\}<\infty.$$
\end{Th}

\section{A necessary condition for the quasi-lisse property}

A finitely strongly generated vertex algebra $V$ is called {\em conical}
if it is conformal,
and
$L_0$ gives a
$\frac{1}{m}\Z_{\geq 0}$-grading 
$$V=\bigoplus_{\Delta\in \frac{1}{m}\Z_{\geq 0}}V_{\Delta}$$ on 
$V$ 
for some $m\in \N$,
$\dim V_{\Delta}<\infty$ for all $\Delta$,
and $V_0=\C$,
where
$V_{\Delta}=\{v\in V\mid L_0v=\Delta v\}$.
Note that if  $V$ is a vertex operator algebra,
that is, if $V$ is integer-graded,
then a conical vertex operator algebra is the same as a vertex operator algebra of CFT type.

Let $V$ be a conical vertex algebra.
The $L_0$-grading induces the grading
$$R_V=\bigoplus_{\Delta\in \frac{1}{m}\Z_{\geq 0}}(R_V)_{\Delta},\quad (R_V)_{0}=\C,$$
on $R_V$.
In other words,
the $L_0$-grading induces a contracting $\C^*$-action on
the associated variety $X_V$.

\begin{Rem}
The 
associated variety of a simple conical quasi-lisse vertex algebra is conjecturally irreducible (\cite{AraMor16}). 
The validity 
of this conjecture implies that
the associated variety of a quasi-lisse vertex algebra is
actually {\em symplectic},
that is,
$X_V$ is the closure of a symplectic leaf.
\end{Rem}

\begin{Rem}
Conical lisse ($C_2$-cofinite) conformal vertex algebras are quasi-lisse,
since $X_V$ is a point in this case.
\end{Rem}

\begin{Pro}\label{Th:nilp1}
Let $V$ be a conical quasi-lisse vertex algebra.
Then the image $[\omega]$ of the conformal vector $\omega$ of $V$ is nilpotent in the Zhu's $C_2$-algebra $R_V$
of $V$.
\end{Pro}
\begin{proof}
Since $V$ is conical,
the $\C^*$-action $\rho$ on $X_V$
induced by the conformal grading
contracts to a point, say $0$,
that is, $\lim\limits_{t\ra 0}\rho(t)x=0$ for all $x\in X_V$.

 Set $z=[\omega]$. It is sufficient  to show that the value of $z$ at any closed point $x$ is zero.
Pick an irreducible component  $Y$ 
of $X_V$ containing $x$.
Note that $Y$ is $\C^*$-invariant, and hence, $0\in Y$.
On the other hand,  there exists a symplectic leaf $\mathscr{L}\subset X_V$ such that
 $Y=\overline{\mathscr{L}}$, the Zariski closure of  $\mathscr{L}$.
 Since it belongs to the Poisson center of $R_V$,
 $z$ belongs to the Poisson center of $\mathcal{O}(\mathscr{L})$.
Hence $z$ is constant on $\mathscr{L}$ as  $\mathscr{L}$ is symplectic,
and thus,
so is 
on $Y$.
Therefore
the value of 
$z$  at $x$ is  the same as the one at
$0$,
which is clearly zero.
\end{proof}

\section{Finiteness of ordinary representations}

Recall that 
a~weak $V$-module $(M,Y_M)$ is called {\it ordinary} 
if $L_0$ acts semi-simply on $M$, any $L_0$-eigenspace  $M_\Delta$ of $M$ of eigenvalue $\Delta\in\mathbb{C}$ is finite-dimensional,
and for any $\Delta\in \mathbb{C}$, we have $M_{\Delta-n}=0$ for all sufficiently large $n\in\mathbb{Z}$.
\begin{Th}
Let $V$ be a quasi-lisse conformal vertex algebra. 
Then the number of the simple ordinary $V$-modules is finite.
\end{Th}
\begin{proof}
Let $A(V)$ be the Zhu's algebra of $V$.
By the Zhu's theorem \cite{Z},
it is sufficient to show that the number of the simple finite-dimensional $A(V)$-modules is finite.

The algebra $A(V)$ is naturally filtered:
There is a natural filtration of $G_{\bullet}A(V)$ induced by the 
filtration $\bigoplus\limits_{\Delta\leq p}V_\Delta$ of $V$ (\cite{Z})
that makes  the associated graded  $\gr _GA(V) $
a 
Poisson algebra.
Moreover,
there is  a  surjection map
$$R_V\ra \gr A(V)$$
of Poisson algebras \cite{De-Kac06,ALY}.
Therefore,
$\on{Specm}(\gr_G A(V))$  is a Poisson subvariety of $X_V$,
and hence has a finitely many symplectic leaves.
Hence, thanks to 
 Theorem 1.4 of \cite{EtiSch10} that follows from Theorem \ref{Th:ES},
 we conclude that $A(V)$ has only finitely many simple finite-dimensional representations.
\end{proof}

\section{Modular linear differential equations}
Let $\vartheta_k$ denote the Serre derivation of weight $k$
\[
\vartheta_k(f)=q\frac{d}{dq}f-\frac{k}{12}E_2 f,
\]
where $E_n(\tau)$ is the normalized Eisenstein series of weight $n\geq 2$.
Let $\vartheta_k^i=\vartheta_{k+2(i-1)}\circ \cdots \circ \vartheta_{k+2}\circ \vartheta_k$ be the $i$-th iterated Serre derivation of weight $k$ with $\vartheta^0_k=1$.
Recall that a~modular linear differential equation (MLDE) of weight $k$ is a~linear
differential equation $$\vartheta_k^n f+\sum_{j=0}^{n-1}P_{j}\vartheta_k^jf=0$$
with a~classical modular function $P_{j}$ of weight $2n-2j$ for each $0\leq j\leq n-1$.

In this section, we prove the following theorem.
Let $\mathbb{H}$ denote the complex upper-half plane.  

\begin{Th}\label{Th:main} 
Let $V$ be a quasi-lisse vertex operator algebra,
$c\in \C$ the central charge of $V$. 
Then the normalized character 
$$\chi_V(\tau)=\on{tr}_V (e^{2\pi i\tau(L_0-c/24)})\qquad (\tau\in\mathbb{H})$$
satisfies a modular linear differential equation of weight $0$.
\end{Th}

Let $(V, Y(\cdot,z),\omega)$ be a~quasi-lisse conformal vertex operator algebra
with the weight grading $V=\bigoplus_{\Delta=0}^\infty V_\Delta$.  Y.~Zhu introduced a~second vertex operator algebra $(V,Y[\cdot,z],\tilde{\omega})$
associated to $V$ (\cite{Z}), where
 the vertex operator $Y[\cdot,z]$ is defined by
linearly extending the assignment 
\[
Y[v,z]=Y(v,e^z-1)e^{z\Delta }=\sum_{n\in\mathbb{Z}}v[n]z^{-n-1}, \quad (v\in V_\Delta,\ \Delta\geq 0),
\]
and 
$\tilde{\omega}=\omega-c/24$. 
We write $Y[\tilde{\omega},z]=\sum_{n\in\mathbb{Z}}L[n]z^{-n-2}$ and
$V_{[\Delta]}=\{v\in V\,|\,L[0]v=\Delta v\}$ for every $\Delta\in \mathbb{Z}_{\geq 0}$.

Set
$A=\mathbb{C}[\tilde{G}_4(q),\tilde{G}_6(q)]$, $V_A=V\otimes A$.
Here, $q$ is a~formal variable and $\tilde{G}_n(q)=\sum_{j=0}^\infty a_{n,j}q^j$,
where $\{a_{n,j}\}_{j\geq 0}$ ($n\geq 2$) are the Fourier coefficients of the Eisenstein series $G_n(\tau)$ of weight $n$, that is,
$G_n(\tau)=\sum_{j=0}^\infty a_{n,j} e^{2\pi i\tau j}$.
Let
 $O_q(V)$ be the
$A$-submodule of $V_A$
generated by
\[
a[-2]b+\sum_{k=2}^\infty (2k-1)\tilde{G}_{2k}(q)a[2k-2]b\qquad (a,b\in V).
\]
Let $[V_A,V_A]$ denote
the $A$-span of elements $a[0]b$, $a,b\in V$, in $V_A$.

\begin{Pro}\label{Pro:key}
The $A$-module
$$V_A/([V_A,V_A]+O_q(V))$$ is finitely generated.
\end{Pro}
\begin{proof}
We set
$$U:=[V_A,V_A]+O_q(V).$$
Also,
put
$R=R_V=V/C_2(V)$,  $C_2(V)_A=C_2(V)\otimes_\C A$, and
\begin{align*}
R_A=R\*_\C A=V_A/C_2(V)_A.
\end{align*}

Define an increasing filtration $G_{\bullet}V_A$ of the $A$-module $V_A$ by
\begin{align*}
G_p V_A=\bigoplus_{\Delta\leq p}V_\Delta\* A.
\end{align*}
This induces the filtration on $U$ and $V_A/U$:
 $G_p U=U\cap G_p V_A$,
$\gr_G U=\bigoplus_p G_p U/G_{p-1}U$,
$G_p (V_A/U)=G_p V_A/G_p U$,
and
$$\gr_G (V_A/U)=\bigoplus_p G_p (V_A/U)/G_{p-1} (V_A/U)
= V_A/\gr_G U.$$

Since
$$\gr_G U\supset \{R_A,R_A\}+C_2(V)_A,$$
we have a surjective map of $A$-modules
$$R_A/\{R_A,R_A\}\twoheadrightarrow 
V_A/\gr_G U=\gr_G 
(V_A/U).$$
As 
$R_A/\{R_A,R_A\}=(R/\{R,R\})\otimes_\C A$,
the assertion follows from 
Theorem \ref{Th:ES}.
\end{proof}

Since $A$ is a~Noetherian ring,
it follows from  Proposition \ref{Pro:key} that 
$V_A/([V_A,V_A]+O_q(V))$ is a~Noetherian $A$-module.
Hence, 
 we have the following lemma.
\begin{Lem}\label{Th:kankei1}
For an element  $a$ of $V$,
there exist $s\in\mathbb{Z}_{\geq 0}$ and $g_i(q)\in A$
{\rm(}$0\leq i\leq s-1${\rm)} such that
\[
L[-2]^sa+\sum_{i=0}^{s-1} g_i(q) L[-2]^ia\in [V_A,V_A]+O_q(V).
\]
\end{Lem}

Let $M$ be an~ordinary $V$-module.
Define the zero-mode action $o(\cdot):V\rightarrow \mathrm{End}(M)$
by linearly extending the assignment
\[
o(a)=\mathrm{Res}_z Y^M(a,z)z^{\Delta-1}dz:M\rightarrow M, \quad (a\in V_\Delta,\ \Delta\geq 0),
\]
For any $a\in V$,  define the formal $1$-point function $\tilde{\chi}_M(a,q)$ by
 $$\tilde{\chi}_M(a,q)=\mathrm{tr}|_M o(a)q^{L_0-c/24}.$$ 
For each $q$-series $f(q)$ and $k\geq 0$,
define the formal Serre derivation $\partial_k$ of weight $k$ by
\[
\partial_k f=q\frac{d}{dq}f(q)+k\,\tilde{G}_2(q)f(q).
\]
Let $\Delta$ and $\ell$ be non-negative integers.
For each $a\in V_{[\Delta]}$ and $f(q)\in  \mathbb{C}[\tilde{G}_2(q),\tilde{G}_4(q),\tilde{G}_6(q)]$ of weight $n$, define the formal iterated Serre derivation $\partial^i$ by
\begin{align*}
\partial^i (f(q) \tilde{\chi}_M(a,q))=\partial_{\Delta+\ell+2i-2}(\partial^{i-1}(f(q) \tilde{\chi}_M(a,q))) \quad (i\geq 1),
\end{align*}
and $\partial^0=\mathrm{id}$.
Here, $f(q)$ is said to be of weight $n$
if $f(q)$ is a~homogeneous element of weight $n$ of the graded algebra $\mathbb{C}[\tilde{G}_2(q),\tilde{G}_4(q),\tilde{G}_6(q)]$, 
where the weight of $\tilde{G}_k(q)$ is $k$ for each $k=2,4,6$.

\begin{Lem}[{cf.~\cite[(5.9)]{DLM00}}]\label{Th:vir1}
Let $v$ be an~element of $V$ and $M$ an~ordinary $V$-module.
We have
\[
\tilde{\chi}_M(L[-2]v,q)=\partial \tilde{\chi}_M(v,q)+\sum_{\ell=2}^\infty \tilde{G}_{2\ell}(q)\tilde{\chi}_M(L[2\ell-2]v,q).
\]
\end{Lem}
 
\begin{proof}
The assertion follows by~\cite[Proposition~4.3.5]{Z} with $a=\omega$ and $b=v$.
\end{proof}

Let $a$ be a primiary vector of $V$ of weight $\Delta$,
that is, $L[0]a=\Delta a$ and 
$L[n]a=0$ for $n\geq 1$.

\begin{Lem}[{cf.~\cite[Lemma~6.2]{DLM00}}]\label{Th:vir2}
For each $i\geq 1$, there exist elements $f_j(q)\in A$ {\rm(}$0\leq j\leq i-1${\rm)} such that for any ordinary $V$-module $M$,
\begin{equation}\label{eqn:serre}
\tilde{\chi}_M(L[-2]^ia,q)=\partial^i\tilde{\chi}_M(a,q)+
\sum_{j=0}^{i-1}f_j(\tau)\partial^j\tilde{\chi}_M(a,q).
\end{equation}
\end{Lem}

\begin{proof}
The proof is similar to that of~\cite[Lemma~6.2]{DLM00}.
We prove the assertion by induction on $i$.
When $i=1$, it follows from Lemma~\ref{Th:vir1} that
$\tilde{\chi}_M(L[-2]a,q)=\partial\tilde{\chi}_M(a,q)$, and therefore~\eqref{eqn:serre} follows. 
Suppose that $i\geq 2$.
Then by~ Lemma \ref{Th:vir1}, we see that
\begin{align*}
&\tilde{\chi}_M(L[-2]^i a,q)\\&=\partial \tilde{\chi}_M(L[-2]^{i-1}a,q)+\sum_{k=2}^\infty
\tilde{G}_{2k}(q) \tilde{\chi}_M(L[2k-2]L[-2]^{i-1}a,q).
\end{align*}
By using the relation of the Virasoro algebra, we have
$$L[2k-2]L[-2]^{i-1}a=c_k \cdot L[-2]^{i-k}a$$
with  a~scalar $c_k$ for each
$2\leq k\leq i$ and $L[2k-2]L[-2]^{i-1}a=0$ if $k\geq i+1$.
 Therefore,
 \begin{align*}
\tilde{\chi}_M(L[-2]^i a,q)=\partial \tilde{\chi}_M(L[-2]^{i-1}a,q)+\sum_{k=2}^{i}
c_k \tilde{G}_{2k}(q) \tilde{\chi}_M(L[-2]^{i-k}a,q).
\end{align*}
By the induction hypothesis, we have~\eqref{eqn:serre}, which completes the proof.
\end{proof}

Let $u$ and $v$ be elements of $V$.

\begin{Lem}[{\cite[Proposition 4.3.6]{Z}}]\label{Th:kankei3}
For every ordinary $V$-module $M$,
\[
\tilde{\chi}_M(u[0]v,q)=0,
\]
\[
\tilde{\chi}_M(u[-2]v,q)+\sum_{k=2}^\infty (2k-1)\tilde{G}_{2k}(q)
\tilde{\chi}_M(u[2k-2]v,q)=0.
\]
\end{Lem}

\begin{Th}\label{Th:abs}
Let $V$ be a quasi-lisse vertex operator algebra,
$a\in V$ primary with $L[0]a=\Delta a$.
For each ordinary $V$-module $M$, the series $\tilde{\chi}_M(a,q)$ converges absolutely and 
uniformly in every closed subset of the domain $\{q\,|\,|q|<1\}$, and the limit function $\chi_M(a,q)$
has the form $q^hf(q)$ with some analytic function $f(q)$ in  $\{q\,|\,|q|<1\}$.
Moreover, the space spanned by $\chi_M(a,q)$ for all ordinary $V$-module $M$ is a~subspace of the space of the solutions of a~modular linear differential equation
of weight $\Delta$.
\end{Th}

\begin{proof}
The proof is similar to those of~\cite[Theorem 4.4.1]{Z} and~\cite[Lemma 6.3]{DLM00}.
By Lemma~\ref{Th:kankei1}, we have $L[-s]^sa+\sum_{i=0}^{s-1}g_i(q)L[-2]^ia\in O_q(V)$ where $s\in \mathbb{Z}_{\geq 0}$ and
$g_i(q)\in A$  for each $0\leq i\leq s-1$.
It then follows by the definition of $O_q(V)$ and Lemma~\ref{Th:kankei3}
that
$\tilde{\chi}_M(L[-2]^sa+\sum_{i=0}^{s-1} g_i(q)L[-2]^ia,q)=0$.
By Lemma~\ref{Th:vir2}, we obtain a~differential equation 
\begin{equation}\label{eq:diffproof}
\partial^s \tilde{\chi}_M(a,q)+\sum_{i=0}^{s-1}h_i(q)\partial^i\tilde{\chi}_M(a,q)=0.
\end{equation}
for the formal series
$\tilde{\chi}_M(a,q)$ with $h_i(q)\in A$.
Since $h_i(q)$ converges absolutely and uniformly
on every closed subset of  $\{q\,|\,|q|<1\}$, and~\eqref{eq:diffproof} is
regular, it follows that $\tilde{\chi}_M(a,q)$ converges uniformly
on every closed subset of $\{q\,|\,|q|<1\}$.
By using~\eqref{eq:diffproof} again, we see that
 the space spanned by $\chi_M(a,q)$ for all ordinary $V$-module $M$ is a~subset of the space of the solutions of the MLDE
$$
\vartheta^s_{\Delta} \chi_M(a,q)+\sum_{i=0}^{s-1}p_i(q)\vartheta^i_{\Delta}\chi_M(a,q)=0,
$$
where  $p_i(q)\in\mathbb{C}[G_4(\tau),G_6(\tau)]$ is the limit function of $h_i(q)$ with $q=e^{2\pi i\tau}$ and $\tau\in\mathbb{H}$. 
The remainder of the theorem is clear.
\end{proof}

Theorem~\ref{Th:main} follows from Theorem~\ref{Th:abs} with $a=|0\rangle$, the vacuum vector of $V$.

\section{Examples of quasi-lisse vertex algebras}
 Let 
 $V^k(\fing)$ be the universal affine vertex algebra associated with a simple Lie algebra 
 $\fing$ at level $k\in \C$,
 and let
 $V_k(\fing)$ be the unique simple graded quotient of $V^k(\fing)$.
 We have $X_{V^k(\g)}=\g^*$,
 where $\g^*$ is equipped with the Kirillov-Kostant-Souriau Poisson structure,
 and
  $X_{V_k(\fing)}$ is a conic, $G$-invariant, Poisson subvariety of $X_{V^k(\g)}=\fing^*$, 
 where $G$ is the adjoint group of $\fing$ (see \cite{Ara12}).
 
 Let $\mc{N}=\{x\in\fing\mid \ad x \text{ is nilpotent}\}$,
 the nilpotent cone of $\fing$,
 which is identified with the zero locus 
 of the argumentation ideal $\C[\fing^*]^G_+$ of 
 the invariant ring $\C[\fing^*]^G$ via the identification $\fing=\fing^*$.
 It is well-known since Kostant \cite{Kos63}
 that 
 the number of $G$-orbits in  $\mc{N}$
 is finite. 
 \begin{Lem}
The affine vertex algebra $V_k(\fing)$ is quasi-lisse if and only if 
the associated variety
$X_{V_k(\fing)}\subset \mc{N}$.
\end{Lem}
\begin{proof}
The ^^ ^^ if" part
is clear
since the symplectic leaves in $\fing^*$
 are the coadjoint  orbits of $G$.
 Conversely,
 suppose that $X_{V_k(\fing)}\not \subset\mc{N} $.
 Since  $X_{V_k(\fing)}$ is closed,
there exists a nonzero semisimple element $s$ in $X_{V_k(\fing)}$.
As it  is conic,
$X_{V_k(\fing)}$ contains  infinitely many orbits of the form
$G.\lam s$, $\lam\in \C^*$.
\end{proof}

Recall that
$V_k(\fing)$ is called {\em admissible} if it is an admissible representation (\cite{KacWak89})
as a module over the affine Kac-Moody algebra $\affg$
associated with $\fing$.
All the admissible affine vertex algebras are quasi-lisse,
since 
their associated varieties are  contained in $\mc{N}$
 (\cite{FeiMal97,Ara09b}).
In fact,
 the associated variety of an admissible affine vertex algebra $V_k(\fing)$
is irreducible, that is,
$X_{V_k(\fing)}=\overline{\mathbb{O}}$ for some nilpotent orbit $\mathbb{O}$
of $\fing$ (see \cite{Ara09b} for the explicit description of the orbit $\mathbb{O}$).

Highest weight representations
of an admissible affine vertex algebra $V_k(\fing)$
are exactly the admissible representations $L(\lam)$ of $\affg$ of level $k$
whose integral Weyl groups are
obtained from  that of $V_k(\fing)$
by an element of the extended affine Weyl group
 (\cite{AdaMil95, A12-2}).
Let  $\mf{h}$ be the Cartan subalgebra of $\mf{g}$.
The modular invariance of the normalized {\em full} characters  
$$e^{2\pi kt}\on{tr}_{L(\lam)}(q^{L_0-c/24}e^{2\pi i x}),\quad (\tau,x,t)\in Y,
$$
of those representations,
where $Y$ is some domain in $\mathbb{H}\times \mf{h}\times \C$,
 has been known since 
Kac and Wakimoto 
 \cite{KacWak88,KacWak89},
 and was extended in \cite{AEkeren} to 
 that of the general (full) trace functions.
  Here it is essential to consider their full characters,
 since an admissible representation is not an ordinary representation in general,
 and thus,
the normalized character 
 $\on{tr}_V (e^{2\pi i\tau(L_0-c/24)})$ is not always well-defined.  
 
Theorem \ref{Th:abs}
states the modular invariance of 
the normalized character
(instead of the normalized full character)
of
an admissible representation that is ordinary.
As far as the authors know, this fact is new.

\smallskip

Here are more examples of quasi-lisse affine vertex algebras.
\begin{Th}[\cite{AM15}]\label{Th:AM}
Assume that $\fing$ belongs to the Deligne exceptional  series
$$A_1\subset A_2\subset G_2\subset D_4\subset F_4\subset E_6\subset E_7\subset E_8,$$
and let $k=-h^{\vee}/6-1$.
Then $X_{V_k(\fing)}\cong \overline{\mathbb{O}_{min}}$,
where $\mathbb{O}_{min}$ is the minimal nilpotent orbit of $\fing$.
\end{Th}

In Theorem \ref{Th:AM},
the affine vertex algebra $V_{-h^{\vee}/6-1}(\fing)$ is admissible
 for types $A_1$,  $A_2$, $G_2$, $ F_4$, 
and so the statement is contained in \cite{Ara09b}. 
However, $V_{-h^{\vee}/6-1}(\fing)$  is not admissible for types $D_4$, $E_6$, $E_7$, $E_8$.
These non-admissible quasi-lisse affine vertex algebras have appeared in \cite{BeeLemLie15}
as main examples of chiral algebras coming from 4d SCFTs. 
In fact, the labels $D_4$, $E_6$, $E_7$, $E_8$ also appear in Kodaira's classification of 
isotrivial elliptic fibrations, 
and the corresponding 4d SCFTs are obtained by applying the $F$-theory to these isotrivial elliptic fibrations.
By construction \cite{BeeLemLie15}, 
 the character of the above non-admissible quasi-lisse affine vertex algebras
are the (homogeneous) Schur indices of these 4d SCFTs obtained from elliptic fibrations.
In mathematics,
such a non-admissible affine vertex algebra was first extensively studied in \cite{Per13}.

In the next section
we derive the explicit form of the characters of these non-admissible quasi-lisse affine vertex algebras.

\smallskip

Now let us give examples of quasi-lisse vertex algebras
outside affine vertex algebras.
Let $\W^k(\fing,f)$ be the $W$-algebra associated with $\fing$ and a nilpotent element $f\in \fing$ at level $k$,
defined by the quantized Drinfeld-Sokolov reduction
\begin{align*}
\W^k(\fing,f)=H^0_{DS,f}(V^k(\fing)),
\end{align*}
where $H^\bullet_{DS,f}(M)$ denotes the cohomology of the BRST complex with coefficient $M$ associated with 
the Drinfeld-Sokolov reduction with respect to $f$.
This definition  was discovered 
by Feigin and Frenkel \cite{FF90} in the case that $f$ is principal 
as a generalization of Kostant's Whittaker model of the center of $U(\g)$
(\cite{Kos78}),
and was generalized to an arbitrary $f$  by Kac and Wakimoto 
 (\cite{KacRoaWak03}).

By \cite{Ara09b},
the natural surjection $V^k(\fing)\ra V_k(\fing)$
induces a surjective homomorphism
$\W^k(\fing,f)\twoheadrightarrow H^0_{DS,f}(V_k(\fing))$ of vertex algebras,
and moreover,
\begin{align*}
X_{H^0_{DS,f}(V_k(\fing))}\cong X_{V_k(\fing)}\cap \mc{S}_f,
\end{align*}
where $\mc{S}_f$ is the Slodowy slice at $f$,
that is,
$\mc{S}_f=f+\fing^e$.
Here $\{e,f,h\}$ is an $\mf{sl}_2$-triple and
$\fing^e$ is the centralizer of $e$ in $\fing$.
Therefore, we have the following assertion.
\begin{Lem}\label{Lem:W}
Let $k$ be non-critical
and suppose that $V_k(\fing)$ is quasi-lisse,
that is, $X_{V_k(\fing)}\subset \mc{N}$.
For any $f\in X_{V_k(\fing)}$,
$H^0_{DS,f}(V_k(\fing))$ is quasi-lisse,
and hence, so is the simple quotient $\W_k(\fing,f)$ of $\W^k(\fing,f)$.
\end{Lem}
We note that 
$H^0_{DS,f}(V_k(\fing))\cong \W_k(\fing,f)$
if $G.f\subset X_{V_k(\fing)}$ 
for type $A$ (\cite{Ara08-a})
and this conjecturally holds for any $\fing$
(\cite{KacWak08}).

Lemma \ref{Lem:W}
implies
admissible affine vertex algebras produce many quasi-lisse $W$-algebras
by applying the Drinfeld-Sokolov reduction.
For instance,
if $k$ is a non-degenerate admissible number (see \cite{A2012Dec}),
then $X_{V_k(\fing)}=\mc{N}$, and hence,
\begin{align*}
X_{H^0_{DS,f}(V_k(\fing))}\cong \mc{N}\cap \mc{S}_f,
\end{align*}
which is  irreducible and therefore symplectic (\cite{Pre02}).
In particular, if $f$ is a subregular nilpotent element in types $ADE$,
$ X_{H^0_{DS,f}(V_k(\fing))}$ has the simple singularity of the same type as $\fing$ (\cite{Slo80}).
In type $A$,
it has been recently shown by Genra \cite{Genra} that
the subregular $W$-algebra $\W^k(\mf{sl}_n,f_{subreg})$
is isomorphic to   Feigin-Semikhatov's $W_n^{(2)}$-algebra (\cite{FeiSem04}) at level $k$ .

\smallskip

See \cite{AM15,AraMor16b} for more examples of quasi-lisse vertex algebras,
and see e.g. \cite{Dan} for more examples of vertex algebras obtained from 4d SCFTs.

\section{The characters of affine vertex algebras associated with the Deligne exceptional series}\label{section:DLE}

In this section, we give the explicit character formulas of the quasi-lisse affine vertex algebras associated with the Deligne exceptional series 
 appeared in Theorem \ref{Th:AM}
by using MLDEs.

Let $\mathfrak{g}$ be a~Lie algebra in the Deligne exceptional series
and $V$ the simple affine vertex algebra $V_k(\mathfrak{g})$ with $k=-h^\vee/6-1$.
The Deligne dimension formula~\eqref{eqn:deligne1} below 
implies that the central charge $c$ of $V$ is
given by $$c=-2h^\vee-2.$$

\begin{Lem}
 The square of the Virasoro element $\omega$ of $V$ is $0$ in the Zhu's $C_2$-algebra $R_V$.
\end{Lem}
\begin{proof}
Let $I$ be the ideal of $R_{V^k(\fing)}=S(\mf{g})$
generated by the image of the maximal submodule of $V^k(\fing)$,
so that 
$R_{V_k(\fing)}=S(\mf{g})/I$.
We need to show that $\Omega^2\in I$,
where $\Omega$ is the Casimir element   of $S(\mf{g})$.

If $\fing$ is not of  type $A$, 
then this result has been already stated in  Lemma 2.1 of \cite{AM15},
see (the proof of) Theorem 3.1 of  \cite{AM15}.
So let $\fing$ be of type $A$,
in which case
the maximal submodule of $V^k(\fing)$
is generated by a singular vector, say $v$ (\cite{KacWak88}).
For 
$\fing=\mf{sl}_2$,
 the assertion follows immediately from a result in \cite{FeiMal97},
 which says that the  image of $v$ in $I$
coincides with $\Omega e$ up to nonzero constant multiplication,
see  the proof of Theorem 4.2.1 of \cite{FeiMal97}.
Finally let $\fing=\mf{sl}_3$.
Then 
the vector $v$ has degree $2$,
cf. ~ \cite{Per08}.
Let $V$ be the $\fing$-submodule of $S^2(\fing)$ generated by the image $[v]$ of $v$ in $I$.
Proposition 3.3 of  \cite{GanSav04} (which is valid for type $A$ cases as well)
says that $\fing \cdot V\subset S^3(\fing)$ contains a submodule isomorphic to $\fing$.
On the other hand,
Kostant's Separation Theorem  (\cite{Kos63}, cf. Proposition 3.2 of \cite{GanSav04}) implies that 
$\fing \cdot \Omega$ is the unique submodule of $ S^3(\fing)$ isomorphic to $\fing$.
Thus,
$\fing \cdot \Omega\subset  \fing  \cdot V$,
and the assertion follows.
\end{proof}
As $(V,Y[\cdot,z])$ is isomorphic to $(V,Y(\cdot,z))$,
it follows that
$$L[-2]^2|0\rangle =
\sum_{i=1}^\ell b_i[-2]c_i $$
for some
$\ell\geq 0$,
where $b_i$ and $c_i$, $1\leq i\leq \ell$, are $L[0]$-homogeneous elements of $V$ 
such that $L[0](b_i[-2]c_i)=4b_i[-2]c_i$.
On the other hand,
it  follows from the definition of $O_q(V)$ that 
$$b_i[-2]c_i\equiv -\sum_{k=2}^\infty f_k(q) b_i[2k-2]c_i\pmod{O_q(V)}$$
for  $1\leq i\leq \ell$.
However,
 $b_i[2]c_i\in\mathbb{C}|0\rangle$ and  $b_i[2k-2]c_i=0$ for $k\geq 3$
 as the $L[0]$-weight of $b_i[-2]c_i$ is $4$.
Therefore,
we get that $$L[-2]^2|0\rangle+g(q)|0\rangle \in O_q(V)$$ with 
$g(q)\in A$.
By using Lemma~\ref{Th:vir2},
we see that the formal characters $\tilde{\chi}_M(|0\rangle,q)$ of all ordinary $V$-modules $M$ satisfy 
a~second order differential equation of the form 
$\partial^2\tilde{\chi}_M(|0\rangle ,q)+f_1(q)\partial \tilde{\chi}_M(|0\rangle ,q)+
f_2(q)\tilde{\chi}_M(|0\rangle ,q)=0$ with $f_1(q),f_2(q)\in A$.
Hence, the characters $\chi_M(\tau)$ of the ordinary $V$-modules $M$
satisfy a~second order MLDE
$L(f)=0$ of weight 0.

The second order MLDEs (of weight 0) have the form
\begin{equation}\label{eqn:kz}
f''(\tau)-\frac{1}{6}E_2(\tau)f'(\tau)-\frac{k(k+2)}{144}E_4(\tau)f(\tau)=0
\end{equation}
with $k\in \C$.
Here, 
\[
'=q\frac{d}{dq}=\frac{1}{2\pi i}\frac{d}{d\tau}.
\]
A~function $f(\tau)$ is called of {\it vacuum type} if $f$ has the form
$f(\tau)=q^{-\alpha/24}(1+\sum_{n=1}^\infty a_nq^n)$ with $\alpha\in \Q$ and $a_n\in \mathbb{Z}_{\geq 0}$ for each $n\geq 1$, where $q=e^{2\pi i \tau}$.
Let $f$ be a~solution of~\eqref{eqn:kz} of the form
$f(\tau)=q^{-\alpha/24}(1+O(q))$ with $\alpha\in \Q$.
Then by substituting $f$ into~\eqref{eqn:kz}, we see that
$\alpha=-k/12$ or $(k+2)/12$.
If $\alpha=-k/12$, it follows that $k$ is one of the following numbers~\cite[(3.12)]{KNS}:
\begin{equation}\label{cand1}
k=\frac{1}{5},\frac{1}{2},1,\frac{7}{5},2,\frac{13}{4},3,\frac{7}{2},\frac{19}{5},4.
\end{equation}
On the other hand, if $\alpha=(k+2)/12$, we have~\cite[(3.16)]{KNS}
\begin{equation}\label{cand2}
k=\frac{1}{5},\frac{1}{2},1,2,3,5,8,11,17,23,29,53.
\end{equation}

Since $V$ is of CFT-type with the central charge $-h^\vee-2$, the character $\chi_V(\tau)$ is of vacuum type and has the form $\chi_V(\tau)=q^{(h^\vee+1)/12}(1+O(q))$.
Therefore,~\eqref{cand1} and~\eqref{cand2} 
imply that the MLDE $L(f)=0$ must be
the following one:
\begin{equation}\label{eqn:diff1}
f''(\tau)-\frac{1}{6}E_2(\tau)f'(\tau)-\frac{(h^\vee-1)(h^\vee+1)}{144}E_4(\tau)f(\tau)=0.
\end{equation}

The vacuum type solutions of~\eqref{eqn:diff1} are also given in~\cite{KK} and~\cite{KNS}.
As a result, we conclude that 
\begin{align*}
&\chi_{V_{-4/3}(A_1)}=\frac{\eta(3\tau)^3}{\eta(\tau)^3},\quad
\chi_{V_{-3/2}(A_2)}=\frac{\eta(2\tau)^8}{\eta(\tau)^8},\quad
\chi_{V_{-5/3}(G_2)}=\frac{E_1^{(3)}(\tau)\eta(3\tau)^6}{\eta(\tau)^8},\\
&\chi_{V_{-2}(D_4)}=\frac{E_4'(\tau)}{240\eta(\tau)^{10}},\quad
\chi_{V_{-5/2}(F_4)}=\frac{E_2^{(2)}(\tau)\eta(2\tau)^{24}}{\eta(\tau)^{28}},\\
&\chi_{V_{-3}(E_6)}=
-\frac{1}{462}\left(\frac{E_6(\tau)E_4'(\tau)}{240\eta(\tau)^{22}}
-\eta(\tau)^2\right),\\
&\chi_{V_{-4}(E_7)}=\frac{1}{204204}\Biggl(
\Delta(\tau)P_2\left(\frac{E_6(\tau)}{\sqrt{\Delta(\tau)}}\right)
\frac{E_4'(\tau)}{240\eta(\tau)^{34}}-\Delta(\tau)\frac{E_6(\tau)}{\eta(\tau)^{34}}
\Biggr),\\
&\chi_{V_{-6}(E_8)}=\frac{1}{38818159380}\Biggl(
\Delta(\tau)^2P_4\left(\frac{E_6(\tau)}{\sqrt{\Delta(\tau)}}\right)
\frac{E_4'(\tau)}{240\eta(\tau)^{58}}-\Delta(\tau)^{5/2}
Q_4\left(\frac{E_6(\tau)}{\sqrt{\Delta(\tau)}}\right)\frac{1}{\eta(\tau)^{58}}
\Biggr).
\end{align*}
Here,
$\eta(\tau)=q^{1/24}\prod_{n=1}^\infty (1-q^n)$,
$\Delta(\tau)=\eta(\tau)^{24}$,
$E_2^{(2)}(\tau)=2E_2(2\tau)-E_2(\tau)$,
\[
E_1^{(3)}(\tau)=1+6\sum_{n=1}^\infty \Bigl(\sum_{d|n}\left(\dfrac{d}{3}\right)(n/d)^2\Bigr)q^n
\]
with the Legendre symbol $\left(\dfrac{\cdot}{\cdot}\right)$,
$P_2(x)=x^2+462$, $P_4(x)=x^4+1341x^2+201894$, and $Q_4(x)=x^3+879x$.

In particular, it follows 
that the characters
of $V_{-4/3}(A_1)$,
$V_{-3/2}(A_2)$,
$V_{-5/3}(G_2)$, and 
$V_{-5/2}(F_4)$  are modular forms,
while those of
 $V_{-2}(D_4)$, $V_{-3}(E_6)$, $V_{-4}(E_7)$ and $V_{-6}(E_8)$ are quasimodular forms of positive depths
(\cite[pp.450]{KNS}).
Moreover, if $h^\vee$ is the dual Coxeter number of $D_4$, $E_6$, $E_7$ or $E_8$, then MLDE~\eqref{eqn:diff1} 
has a~solution with a~logarithmic term (see \cite[section~5]{KK} and~\cite[Remark~3.8]{KNS}).
Note that the  above formula for $A_1$ and $A_2$
follows also from  the recent remark \cite{KW4} by Kac and Wakimoto.

\begin{Rem}\label{Th:remk1}
It should be notable that
the coefficient of $E_4(\tau)f(\tau)$ in~\eqref{eqn:diff1} is a~non-constant rational function in $h^\vee$,
as
such phenomena are 
 often observed for the Deligne exceptional series.
In fact,
the dimensions of specific modules over any Lie algebra in the Deligne exceptional series satisfy
the so-called {\it Deligne dimension formulas}, which are rational functions in $h^\vee$. For example, 
\begin{equation}\label{eqn:deligne1}
\dim \mathfrak{g} = \frac{2(h^\vee +1)(5h^\vee -6)}{h^\vee +6},
\end{equation}
and
$
\dim L(2\theta) = 5(h^\vee)^2(2h^\vee+3)(5h^\vee-6)/(h^\vee+12)(h^\vee+6)
$
(\cite{CdM}, \cite{D} and \cite{LM02}).
Here,
$L(2\theta)$ is the irreducible highest weight module of weight $2\theta$ over $\mathfrak{g}$.
In the vertex algebra setting,
the first example of such phenomena
was first observed
in \cite{T}, where the coefficients of the MLDEs which the characters of the affine vertex algebras $V_1(\mathfrak{g})$ at level~$1$
associated with the Deligne exceptional series satisfy are expressed as rational functions in  $h^\vee$.
The second example of such phenomena was found
in \cite{K}, where 
the minimal affine $\mathcal{W}$-algebras associated with the Deligne exceptional series at level $-h^\vee/6$ were shown to be lisse and rational.
\end{Rem}

\begin{Rem}
It follows from a result of~\cite{KNS} that
there is a~vacuum type solution of~\eqref{eqn:diff1} with $h^\vee=24$.
Although the Coxeter number $h^\vee$ of any Lie algebra in the Deligne exceptional series does not coincide with $24$,
the number ``$h^\vee=24$" appears in many studies of the Deligne exceptional series (see e.g.~\cite{CdM} and \cite{K14}).
\end{Rem}

\newcommand{\etalchar}[1]{$^{#1}$}


\end{document}